\documentclass[11pt]{article}
\usepackage{amssymb,amsthm,amsmath,mathrsfs}
\usepackage{prettyref,cite}
\usepackage{graphicx,tikz,tkz-berge}
\usepackage [a4paper, footskip=1cm, headheight = 16pt, top=2cm, bottom=2.5cm,  right=2cm,  left=2cm]{geometry}

\newrefformat{chp}{\chaptername~\ref{#1}}
\newrefformat{sec}{Section~\ref{#1}}
\newrefformat{subsec}{Subsection~\ref{#1}}
\newrefformat{thm}{Theorem~\ref{#1}}
\newrefformat{lem}{Lemma~\ref{#1}}
\newrefformat{cor}{Corollary~\ref{#1}}
\newrefformat{con}{Conjecture~\ref{#1}}
\newrefformat{prop}{Proposition~\ref{#1}}
\newrefformat{dfn}{Definition~\ref{#1}}
\newrefformat{exm}{Example~\ref{#1}}
\newrefformat{rem}{Remark~\ref{#1}}
\newrefformat{fig}{\figurename~\ref{#1}}
\newrefformat{tbl}{\tablename~\ref{#1}}
\newrefformat{eq}{$($\ref{#1}$)$}
\newrefformat{fac}{Fact~\ref{#1}}

\newtheorem{pro}{Proposition}[section]

\newtheorem{lem}[pro]{Lemma}
\newtheorem{thm}[pro]{Theorem}
\newtheorem{con}[pro]{Conjecture}
\newtheorem{predef}[pro]{Definition}
\newtheorem{prerem}[pro]{Remark}

\newcommand{\bq}{\begin{equation}}
\newcommand{\eq}{\end{equation}}
\newcommand{\ds}{}
\begin{document}
\title{Decompositions of complete  uniform multi-hypergraphs into Berge paths and cycles of arbitrary lengths}
\author{
	Ramin Javadi \thanks{Corresponding author} \thanks{Department of Mathematical Sciences,
		Isfahan University of Technology,
		P.O. Box: 84156-83111, Isfahan, Iran. Email Address: 
		{rjavadi@cc.iut.ac.ir}.}  \and
	Afsaneh Khodadadpour\thanks{Department of Mathematical Sciences, Isfahan University of Technology,
		P.O. Box: 84156-83111, Isfahan, Iran. Email Address: 
		{a.khodadadpour@math.iut.ac.ir}.}  \and
	Gholamreza Omidi\thanks{Department of Mathematical Sciences, Isfahan University of Technology, P.O. Box: 84156-83111, Isfahan, Iran. Email Address: 
		{romidi@cc.iut.ac.ir}.}
	}
\date{}
\maketitle
\begin{abstract}
In 1981, Alspach conjectured  that  the complete graph $ K_{n} $ could be decomposed into cycles of arbitrary lengths, provided that the obvious necessary conditions would hold. This conjecture was proved completely by Bryant, Horsley and Pettersson in 2014. Moreover, in 1983, Tarsi conjectured that the obvious necessary conditions for packing pairwise edge-disjoint paths of arbitrary lengths in the complete multigraphs were also sufficient. The conjecture was confirmed by Bryant in 2010. In this paper, we investigate an analogous problem as the decomposition of the complete uniform multi-hypergraph $ \mu K_{n}^{(k)} $ into Berge cycles and Berge paths of arbitrary given lengths. We show that for every integer $ \mu\geq 1 $, $ n\geq 108 $ and $ 3\leq k<n $, $ \mu K_{n}^{(k)} $ can be decomposed into Berge cycles and Berge paths of arbitrary lengths, provided that the obvious necessary conditions hold, thereby generalizing a result by K\"{u}hn and Osthus on the decomposition of $K_{n}^{(k)}$ into Hamilton Berge cycles. Furthermore, we obtain the necessary and sufficient conditions for packing the cycles of arbitrary lengths in the complete multigraphs.\\[4pt]
\end{abstract}

\section{Introduction}\label{sec:intro}
Throughout the paper, all multigraphs are loopless. 
A multigraph is called \textit{even}, if every vertex has an even degree.
Given a multigraph $ G  $  and a pair of distinct vertices $x,y\in V(G)$,  the \textit{multiplicity} of $ xy $ in $ G $, denoted by $m_G(xy)$, is defined as the number of (parallel) edges between $ x $ and $ y $ in $ G $.
 A \textit{decomposition} of a multigraph $ G $ is a family of subgraphs of $ G $ whose edge sets partition the edge set of $ G $. Also, a \textit{packing} of $G$ is a collection of pairwise edge-disjoint subgraphs of $G$.
 Let  $ M=(m_1,\ldots,m_t) $ be a list of positive integers. An \textit{$M$-path packing} of a multigraph $G$ is a packing $\{G_1,\ldots,G_t\}$ such that for every $i\in\{1,\ldots, t\}$, $ G_i $ is an  $m_i$-path (a path with $m_i$ edges).
Also, an \textit{$M$-path decomposition} of $  G$ is a decomposition $\{G_1,\ldots,G_t\}  $ such that for every $i\in\{1,\ldots, t\}$, $ G_i $ is an  $m_i$-path. An \textit{$M$-cycle packing} and an \textit{$M$-cycle decomposition} can be defined in a similar way by substituting cycles for paths.

The complete multigraph $\lambda K_n$ is a multigraph on $n$ vertices where there are exactly $ \lambda $ edges between each pair of vertices. In 1983, Tarsi \cite{tarsi} conjectured that for a list of integers $ M=(m_1,\ldots,m_t) $,  the obvious necessary conditions for the existence of an $M$-path packing for $\lambda K_n$ were also sufficient. In other words, $\lambda K_n$ admits an $M$-path packing if and only if $1\leq m_i\leq n-1$, for every $i\in\{1, \ldots , t\}$  and $ m_1+\cdots+m_t \leq \lambda n(n-1)/2$. Tarsi \cite{tarsi} proved his conjecture for both cases ``$n$ odd'' and  ``$  \lambda$  even'' with some limitation on the maximum length of paths in $M$. He also proved his conjecture for the special case where all paths would have the same length. A survey on the path decompositions can be found in \cite{heinrich}.  Finally, the problem was completely solved by Bryant in \cite{bryant10} as follows.
\begin{thm} {\rm\cite{bryant10}}\label{thm:path} Let $n,\lambda $ and $ t $ be positive integers and $M=(m_1, \ldots ,m_t)$ be a list of positive integers. Then $\lambda K_n$ admits an $M$-path packing if and only if $m_i \leq n-1$, for every $i\in\{1, \ldots , t\}$ and $ m_1+\cdots+m_t \leq \lambda \binom{n}{2}$.
\end{thm}

Note that if a multigraph $G$ admits a cycle decomposition, then $G$ is even. Thus, when $\lambda (n-1)$ is odd, $\lambda K_n$ has no cycle decomposition and for every cycle packing of $ \lambda K_n $, in the \textit{leave multigraph} (the multigraph obtained from $ \lambda K_n $ by removing all the edges in the packing), every vertex has an odd degree. Now, suppose that $I$ is a perfect matching in $\lambda K_n$, when $ {\lambda(n-1) }$ is odd, and it is empty, when $ {\lambda(n-1) }$ is even. The decomposition of $\lambda K_n-I$ into cycles of the prescribed lengths has been at the center of attention for many years. 
In 1981, Alspach \cite{alspach81} conjectured that for a list of integers $M=(m_1,\ldots, m_t)$, $ K_n-I $ admits an $M$-cycle decomposition if and only if $3\leq m_i\leq n$, for every $i\in\{1, \ldots , t\}$  and $ m_1+\cdots+m_t =  \binom{n}{2}-|I|$. Many efforts have been made aiming to prove this conjecture in past decades.  
The special case where all of the cycles have the same length was proved in \cite{alspach01,sajna02}. Eventually, in 2014,  Bryant et al. \cite{bryant14} using the earlier important results, gave an affirmative answer to Alspach's conjecture.
The multigraph analogue of Alspach's conjecture has also been studied and 
recently, a complete solution for the decomposition of  $\lambda K_n-I$ into cycles of arbitrary lengths has been obtained by Bryant et al. in \cite{bryant15} as follows.

Let $f(\lambda,n)$ be the number of edges of $\lambda K_n-I$, i.e. 
\begin{equation} \label{eq:f}
 f(\lambda,n)=\begin{cases}
\lambda\,\dfrac{n(n-1)}{2} & \text{if } \lambda (n-1) \text{ is even,} \\[4pt]
\lambda\,\dfrac{n(n-1)}{2}-\dfrac{n}{2}  & \text{if } \lambda (n-1) \text{ is odd.}
\end{cases}
\end{equation}

For a list of integers $M$, let $\nu_2(M)$ denote the number of occurrences of $ 2 $ in $ M$ and $ \sigma(M) $ denote the sum of all integers in $ M $. For positive integers $n,\lambda$, a list $M={(m_{1},\ldots,m_{t})} $ of integers is said to be \textit{${ (\lambda,n)}$-admissible} if
\begin{itemize}
\item[(i)] $ 2\leq m_{1},\ldots,m_{t}\leq n $;

\item[(ii)] $ \sigma(M)=f(\lambda,n) $;

\item[(iii)] for $ \lambda $ odd, $ \nu_2(M)\leq  \frac{(\lambda -1)}{2}\binom{n}{2}$; and

\item[(iv)] for $\lambda$ even, $ \max\{m_1, \ldots ,m_t\} + t-2 \leq \frac{\lambda}{2}\binom{n}{2}$.

\end{itemize}
The following theorem asserts that the above conditions are necessary and sufficient for the existence of an $M-$cycle decomposition of $\lambda K_n-I$.
\begin{thm} {\rm \cite{bryant15}}\label{thm2} \label{thm:cycle}
For positive integers $ \lambda, n$ and a list $ M $ of integers, there is an $M$-cycle decomposition of $\lambda K_n-I$ if and only if $M$ is a $(\lambda , n)$-admissible list.
\end{thm}

The conditions (i) and (ii) are obviously necessary for the existence of an $M$-cycle decomposition for $\lambda K_n-I$. The conditions (iii) and (iv) provide some restrictions on the number of cycles of length two and can be deduced by a discussion on the multiplicity of edges in $\lambda K_n-I$ (see \cite{bryant15}).

Towards proving Alspach's conjecture, the problem of packing cycles of the prescribed lengths in $\lambda K_n-I$ has been studied in the literature (e.g. see \cite{bryant08}). In this paper, first as a generalization of Theorem~\ref{thm:cycle}, we prove the necessary and sufficient conditions for the existence an $M$-cycle packing in $\lambda K_n-I$, for a given list of integers $M$. More precisely, we prove the following.

\begin{thm} \label{lem:disjoint}
Let $\lambda$ and $n$ be two positive integers and $M=(m_1,\ldots, m_t)$ be a list of integers, where $m_1=\max\{m_1,\ldots, m_t\}$. Also, let $f(\lambda,n)$ be as in \eqref{eq:f} and define $r=f(\lambda,n)-\sigma(M)$.  Then, $\lambda K_n-I$ admits an $M$-cycle packing if and only if,
\begin{itemize}
\item[\rm (i)] $ 2\leq m_{1},\ldots,m_{t}\leq n $ and $r\geq 0$;
\item[\rm (ii)] for $\lambda$ odd, either $r\neq 1,2$ and $  \nu_2(M)\leq  \frac{(\lambda -1)}{2}\binom{n}{2} $, or $r=2$ and $   \nu_2(M)< \frac{(\lambda -1)}{2}\binom{n}{2} $; and
\item[\rm (iii)] for $\lambda$ even, either $r=0$ and $ m_1+t-2\leq \frac{\lambda}{2}\binom{n}{2}$, or $r\geq 2$ and $  m_1+t-2< \frac{\lambda}{2}\binom{n}{2}$.
\end{itemize}
\end{thm}
Following the same line of thought, the analogous problems on the decompositions of the complete uniform hypergraphs have been studied in the literature.  A (\textit{multi}-)\textit{hypergraph} is a pair $H=(V,E)$, where $V$ is a finite set of vertices and $E$ is a family of subsets of $V$ called edges (an edge can be repeated several times). A hypergraph is called \textit{$k$-uniform} if each edge is of size $k$. The complete $k$-uniform hypergraph on $n$ vertices with multiplicity $ \mu $, denoted by $\mu K_n^{(k)}$, is defined as the hypergraph with vertex set $V$ of size $n$, such that each $k$-subset of $V$ appears exactly $\mu$ times as its edges. Thus, $ \mu K_n^{(k)} $ has exactly $\mu \binom{n}{k}$ edges.
A \textit{Berge path} of length $ \ell $ in a hypergraph $ H $ is an alternating sequence $ v_{1} , e_{1} , v_{2} , \ldots , v_{\ell} , e_{\ell}, v_{\ell+1} $ of distinct vertices $ v_{i} $ and distinct edges $ e_{i} $ of $ H $, such that each edge $ e_{i} $ contains vertices $ v_{i} $ and $ v_{i+1}$. Also, a \textit{Berge cycle} of length $ \ell $ in $ H$ 
 is an alternating sequence $ v_{1} , e_{1} , v_{2} , \ldots , v_{\ell} , e_{\ell} $ of distinct vertices $ v_{i} $ and distinct edges $ e_{i} $ of $ H $, such that each edge $ e_{i} $ contains the vertices $ v_{i} $ and $ v_{i+1}$ (assuming $ v_{\ell+1}=v_1$). The sequence of vertices $ (v_1,\ldots, v_{\ell+1}) $ (resp. $  (v_1,\ldots, v_{\ell})  $) is called the \textit{core sequence} of the Berge path (resp. Berge cycle). Furthermore, a Berge cycle of length $ n $ in a hypergraph with $ n $ vertices is called a \textit{Hamilton Berge cycle}.
A \textit{decomposition} of $ \mu K_n^{(k)} $ into Berge cycles and paths is a family of Berge cycles and paths in $ \mu K_n^{(k)}$ whose edges are disjoint and have the union the edge set of $ \mu K_n^{(k)} $.  

Fix integers $k$ and $n$ such that $3\leq k<n$. If  $K^{(k)}_n$ has a decomposition into Hamilton Berge cycles, then  $n$ should obviously divide the number of edges $\binom{n}{k}$. Bermond et al. \cite{bermond78} conjectured that this condition could also be sufficient for the existence of such a decomposition of $  K^{(k)}_n $.  For $k=3$, the conjecture follows from combining the results of Bermond \cite{bermond2} and Verrall \cite{verrall}. 
Petecki \cite{petecki} added some restrictions  on the decomposition and found the necessary and sufficient conditions for the existence of such decomposition. 
Finally, in 2014, K\"{u}hn and Osthus  \cite{kuhn} proved the conjecture for every $k\geq 3$ as long as $ n $ was not too small. More precisely, they proved the following. 
\begin{thm} {\rm \cite{kuhn}}\label{thm4}
Suppose that  $n\geq 100$, $ 3\leq k<n$ and $n$ divides $\binom{n}{k}$. Then the complete $k$-uniform hypergraph $K^{(k)}_n$  has a decomposition into Hamilton Berge cycles.
 \end{thm}
As the second task in this paper, we investigate the decomposition of complete uniform hypergraphs into Berge cycles and paths with the prescribed arbitrary lengths. In fact, as a generalization of Theorem~\ref{thm4}, we prove the following. 
\begin{thm}\label{thm5}
 Let $\mu , k , n$ be positive integers such that $3\leq k < n $ and if $k= 3$, then $ n\geq 108$, if $k= 4$, then $ n\geq 54 $ and if $ k\geq 5 $, then $ n\geq 38 $. Also, let $  C=(m_{1},\ldots,m_{s})$ and  $P=(n_1,\ldots,n_t)$ be two lists of integers such that $ 2\leq m_{i}\leq n $, for every $ i\in\{1,\ldots,s\} $ and $1\leq n_i\leq n-1$,  for every $ i\in\{1,\ldots,t\} $  and
 $ \sum_{i=1}^{s} m_{i}+\sum_{i=1}^{t} n_{i}=\mu  \binom{n}{k}$.
 Then, $\mu K_n^{(k)} $ has a decomposition into $ s $ Berge cycles of lengths $ m_{1},\ldots,m_{s} $
 and $ t $ Berge paths of lengths $ n_{1},\ldots,n_{t} $.
 \end{thm}

The organization of forthcoming sections is as follows. In Section~\ref{sec:tools}, we prove some preliminary lemmas that will be used later on. In Section~\ref{sec:packing}, we prove Theorem~\ref{lem:disjoint}. Finally, in Section~\ref{sec:hypergraph}, we apply the obtained results in previous sections to prove Theorem~\ref{thm5}.


\section{Tools}\label{sec:tools}
In this section we provide some tools required in the proof of Theorems~\ref{lem:disjoint} and \ref{thm5}. The following lemma will be used in the proof of Theorem~\ref{lem:disjoint}. It is a generalization of a result in \cite{bryant15} and is proved by similar ideas. 

\begin{lem}\label{lem:packing}
Assume that $G$ is a multigraph in which there are even edges between every pair of vertices. Also, let $\mathcal{P}$ be a cycle packing of $G$ and let $C_0\in \mathcal{P}$. Then, $|\mathcal{P}|\leq |E(G)|/2-|E(C_0)|+f(G,{\mathcal{P}})$, where $f(G,{\mathcal{P}})=2$, when $\mathcal{P}$ is a decomposition of $G$ and $f(G,{\mathcal{P}})=1$, otherwise.
\end{lem}
\begin{proof}
The claim is clear when $|\mathcal{P}|=1$. So, suppose that $|\mathcal{P}|\geq 2$. Let $R(G)=E(G)\setminus \bigcup (E(C): C\in \mathcal{P})$.
First, note that $G$ is even and thus, the induced subgraph of $ G $ on $R(G)$ is also even. Therefore, either $|R(G)|=0$ or $|R(G)|\geq 2$. 
 On the contrary, assume that $|\mathcal{P}|> |E(G)|/2-|E(C_0)|+f(G,{\mathcal{P}})$ and $|E(G)|$ is minimal subject to this property. 
First, we claim that $|E(C_0)|\neq 2$. If $|E(C_0)|=2$, then let $\mathcal{P}'=\mathcal{P}\setminus \{C_0\}$ and  $G'=G-E(C_0)$. Also, let $C'\in \mathcal{P}'$. Then 
\[|\mathcal{P}'|+1=|\mathcal{P}|> \frac{|E(G)|}{2}-2+f(G,{\mathcal{P}})\geq \frac{|E(G')|}{2}+1-|E(C')|+f(G',{\mathcal{P}'}),\]
which is in contradiction with the minimality of $G$. Therefore,  $|E(C_0)|\geq 3$. 
 If the induced subgraph of $ G $ on $R(G)$ contains a cycle $C$ of length two, then $\mathcal{P}$ is a cycle packing of  $G'=G-E(C)$ and 
 \[|\mathcal{P}|> \frac{|E(G')|}{2}-|E(C_0)|+1+f(G,{\mathcal{P}})\geq \frac{|E(G')|}{2}-|E(C_0)|+f(G',{\mathcal{P}}),\]
  which is again in contradiction with the minimality of $G$. Thus, the induced subgraph of $G$ on $R(G)$ is simple. Now, we claim that  the induced subgraph of $G$ on $E(G)\setminus (R(G)\cup E(C_0))$ is also simple. On the contrary, assume that there exist two parallel edges $e_1,e_2$, such that for every $i\in \{1,2\}$, $e_i\in C_i$, for some $C_i\in \mathcal{P}\setminus \{C_0\}$. 
Note that the induced subgraph of $ G $ on $ E(C_1)\cup E(C_2)\setminus \{e_1,e_2\} $  admits a cycle decomposition $ \mathcal{P}_0 $. Now, define $G'=G-\{e_1,e_2\}$ and $\mathcal{P}'=\mathcal{P}\cup \mathcal{P}_0\setminus \{C_1,C_2\}$ (note that it is possible that  $C_1=C_2$ and it occurs when $e_1,e_2$ form a cycle of length two in $\mathcal{P}$. In this case, $\mathcal{P}_0$ is empty). Then, $|\mathcal{P}'|\geq |\mathcal{P}|-1$ and $f(G,\mathcal{P})=f(G',\mathcal{P}')$. Therefore, 
\[|\mathcal{P}'|\geq |\mathcal{P}|-1> \frac{|E(G)|}{2}-|E(C_0)|+f(G,{\mathcal{P}})-1= \frac{|E(G')|}{2}-|E(C_0)|+f(G',{\mathcal{P}'}), \]
a contradiction with the minimality of $G$. Hence, the induced subgraphs of $ G $ on $R(G)$ and $E(G)\setminus (R(G)\cup E(C_0))$ are both simple. Thus, in $ G $, there are at most three edges between every pair of vertices and since the multiplicity of every edge of $ G $ is even, the multiplicity of every edge of $G$ is exactly two. 
  Let $C_0'$ be the cycle in $G$ of length $|E(C_0)|$  edge-disjoint from $C_0$ whose edges are parallel with the edges of $C_0$  (which exists since $|E(C_0)|\geq 3$). If $C_0'\in \mathcal{P}$, then  
 \[|E(G)|\geq 2|E(C_0)|+2(|\mathcal{P}|-2)+|R(G)|\geq 2|E(C_0)|+2|\mathcal{P}|-2f(G,\mathcal{P}),\] a contradiction (the last inequality holds since either $|R(G)|=0$ or $|R(G)|\geq 2$). Therefore, $C_0'\not\in \mathcal{P} $. Let $E'=E(G)\setminus (E(C_0)\cup E(C_0'))$. Thus,  every cycle in $\mathcal{P}\setminus\{C_0\}$ contains at least one edge in $E'$. Since the multiplicity of every edge in $E'$ is two and the induced subgraph of $ G $ on $E(G)\setminus (R(G)\cup E(C_0))$ is simple, we have  
 \[|E(G)|-2|E(C_0)|=|E'|\geq 2(|\mathcal{P}|-1)\geq 2(|\mathcal{P}|-f(G,\mathcal{P})),\]
 a contradiction. This proves the lemma. 
\end{proof}

In the proof of Theorem \ref{thm5}, a technical lemma is required whose proof is obtained by a slight modification of the ideas in \cite{kuhn}. First, we set a couple of notations.
 For integers $ n,k $, $ 0 \leq k\leq n $, the set of all $ k $-subsets of $  [n]=\{1,\ldots,n\} $ is  denoted  by $ {[  n ]} ^{(k)} $.
 For a family $  S\subseteq {[  n ]} ^{(k)} $ and an integer $ \ell $, $ 0\leq \ell\leq k $, the \textit{$ \ell $th lower shadow} of $ S $ is defined as the set $ \delta _\ell ^-{(S)}$ consisting of all  $ t\in {[  n ]} ^{(k-l)}$  such that there exists  $s\in S $ with  $ t\subseteq s $. Similarly, for a family $  S\subseteq {[  n ]} ^{(k)} $ and an integer $ \ell$, $ 0\leq \ell\leq n-k $,  the \textit{$ \ell $th upper shadow} of $ S $ is the set $  \delta _\ell ^+{(S)}$ consisting of all  $  t\in {[  n ]} ^{(k+\ell)}$ such that  there exists $  s\in S$ with $ s \subseteq t $.

We also need the following lemma from \cite{kuhn}. Note that the items (i) and (iii) are the same as Lemma~4 in \cite{kuhn} and the item (ii) is directly deduced from its proof.
\begin{lem} {\rm \cite{kuhn}}\label{lem6}
\begin{itemize}
\item[\rm (i)] Let $ k,n \in \mathbb{N} $ be such that $ 3\leq k\leq n $. Given a nonempty $  S\subseteq {[  n ]} ^{(k)}$, define $ s\in \mathbb{R} $ by ${|S| =\binom{s}{k} } $. Then $ {|\delta _{k-2} ^-{(S)}|\geq \binom{s}{2}}$.

\item[\rm (ii)] Suppose that $ {S' \subsetneq {[n]}^{(2)} } $ and let $ c,d \in \mathbb{N}\cup \{0\} $ be such that $ c<n , d< n-(c+1) $ and $ |S'| =  cn- \binom{c+1}{2} +d$. Then $|\delta _1 ^+{(S')}|\geq c \binom{n-c}{2}  +d(n-c-2)-\binom{d}{2} $.

\item[\rm (iii)] If ${ S' \subseteq {[n]}^{(2)} } $ and $ {|S'| \leq n-1  }$, then ${ \delta _2 ^+{(S')} \geq \vert S' \vert \binom{n-|S'| -1}{2} + \binom{|S'|}{2} (n-|S'| -1)}$.
\end{itemize}
\end{lem} 

Now, we are ready to prove the following technical lemma.
 \begin{lem}\label{thm:sdr} Let $k,n\in \mathbb{N}$ be such that $3\leq k\leq n-3 $.
 Also, let $\mu,\lambda_1,\lambda_2\in \mathbb{N}$ be such that ${0\leq \lambda_2 -\lambda_1 \leq 5}$.
 Suppose that $H$ is a multigraph on $n$ vertices with $\mu \binom{n}{k}$ edges  such that for every pair of distinct vertices $ x,y\in V(H) $, the multiplicity of $xy$ satisfies $ { \lambda_1 \leq m_{H}(xy)  \leq \lambda_2} $.  Let $G=(A,B)$ be a bipartite graph such that the part $ A $ is the set of all the edges of $ {\mu K^{(k)}_n }$ and the part $ B $ is the set of all edges of $ H $. Also, a vertex $ K\in A $ is adjacent to a vertex $ xy\in B $  in $ G $ if and only if $ \{x,y\}\subseteq K$. If  either $k= 3$ and $ n\geq 108$, or $k= 4$ and $ n\geq 54 $, or $ k\geq 5 $ and $ n\geq 38 $, then $ G $ contains a perfect matching.
 \end{lem}
  \begin{proof}
 We prove the claim by verifying Hall's condition  for $G$. Consider a nonempty set $ S\subseteq A $. We prove that $ |N_G(S)|\geq |S| $.
   Let $ \alpha =\lambda_2-\lambda_1 $ and $ S_1=S\cap E(K_{n}^{(k)}) $. Define $ s_1 \in \mathbb{R}$ with $ k\leq s_1\leq n $ by $ |S_1|=\binom{s_1}{k}. $
Since for every pair of distinct vertices $x,y\in V(H)$, $ { \lambda_1 \leq m_{H}(xy)  \leq \lambda_2} $, we have 
 \begin{equation} \label{eq:mu}
   \lambda_1\binom{n}{2}  \leq |E(H)|=\mu \binom{n}{k}  \leq \lambda_2 \binom{n}{2}.
 \end{equation}
We consider the following cases.\\

\textbf{Case 1.} $k=3$. Due to the definition of the graph $ G $, Lemma~\ref{lem6}(i) and \eqref{eq:mu}, we have

 \[ {| N_G{(S)}|=| N_G{(S_1)}|\geq \lambda_1  \big| N_G{(S_1)} \cap E( K_n)\big| \geq \big( \dfrac{n-2}{3}\mu-\alpha\big) \binom{s_1}{2}}.\]
  If $ {s_1\leq n-3\alpha/\mu } $, then
  \[| N_G{(S)}|\geq {( \dfrac{n-2}{3}\mu-\alpha) \binom{s_1}{2} \geq \mu \binom{s_1}{3} } \geq |S|, \]
  and we are done. Now, assume that ${ s_1> n-3\alpha/\mu }   $.


 Suppose that $S'=B\setminus N_G(S)\not=\emptyset$ and $S'_{1}=S'\cap E(K_n)\not=\emptyset$.
 By Lemma~\ref{lem6}(i), we have

\begin{align*}
 \rvert S_{1}'\rvert &=\binom{n}{2} -\big| N_G{(S_1)} \cap E( K_n) \big| \leq \binom{n}{2}- \binom{s_1}{2} <  \binom{n}{2} -  \binom{n-3\alpha/\mu}{2} \\
 &\leq \binom{n}{2}- \binom{n-3\alpha}{2} \leq 3\alpha n - \dfrac{9\alpha ^2+3\alpha}{2}.
 \end{align*}

Let  $ c,d \in \mathbb{N}\cup \{0\} $ be such that $ c<n , d< n-(c+1) $ and ${ |S_1'| =  cn- \binom{c+1}{2} +d}$ (such numbers always exist, to see this it is enough to take $ c $ as the maximum number $ e $ satisfying $en- \binom{e+1}{2}\leq  |S_1'| $). Thus,
 $$ {cn- \binom{c+1}{2}  +d  <3\alpha n - \dfrac{9\alpha ^2+3\alpha}{2}}, $$
 and hence,
 $$ {(c-3\alpha ) (c-2n+3\alpha+1) >0}.$$
  This implies that  either $ c<3\alpha $ or $ c> 2n-3\alpha -1 $. Since $ c<n $, $\alpha\leq 5$ and $n\geq 16$, the former case occurs and we have  $ c\leq 3\alpha-1 $.
 By lemma~\ref{lem6}(ii),
  \begin{equation}\label{eq:1}
 {\arrowvert N_G{(S')}\arrowvert =  \arrowvert N_G{(S_{1}')}\arrowvert \geq \mu \left(c \binom{n-c}{2} +d(n-c-2)- \binom{d}{2}\right)}.
  \end{equation}
 Now, assume that the following two inequalities hold.
 \begin{equation}\label{eq:2}
   \mu c  \binom{n-c}{2}  > (\dfrac{n-2}{3}\mu +\alpha) cn,
  \end{equation}
 \begin{equation}\label{eq:3}
    \mu(d(n-c-2)- \binom{d}{2}) > (\dfrac{n-2}{3}\mu +\alpha) d.
  \end{equation}
Then, by \eqref{eq:1}, \eqref{eq:2} and \eqref{eq:3}, we have

\[ |N_G(S')| > (\dfrac{n-2}{3}\mu +\alpha)(cn+d)\geq \lambda_2 \rvert S_{1}'\rvert \geq \rvert S' \rvert, \]
and therefore, $ |N_G(S)|\geq |S| $. On the other hand, note that since $ c\leq 3\alpha -1 $, \eqref{eq:2} holds if
 \[ \mu \binom{n-3\alpha+1}{2}  > (\dfrac{n-2}{3}\mu +\alpha)n,
 \]
 and then if
\[{ n^{2} +(-24\alpha+7)n+27\alpha^{2}-9\alpha} >0,
\]
which holds for $ n\geq  108$ and $ \alpha \leq 5 $.
 Also, since $ d< n-c-1 $, \eqref{eq:3}  holds if
 $ \mu({n-c-2})/{2} > ({n-2})\mu/3 +\alpha $ and since $ c\leq 3\alpha-1 $, it holds if $ {n\geq 15\alpha-1}$. 
 Hence, in this case, Hall's condition holds and $G$ contains a perfect matching. \\

 \textbf{Case 2. } $k=4. $ By Lemma~\ref{lem6}(i) and \eqref{eq:mu}, we have

\[ |N_G{(S)}|=|N_G{(S_1)}| \geq \lambda_1 | N_{G}{(S_1)} \cap E( K_n) | \geq (\mu \frac{(n-2)(n-3)}{12}-\alpha) \binom{s_1}{2}.\]
Note that if
\begin{equation}\label{eq:4}
(\mu \dfrac{(n-2)(n-3)}{12}-\alpha  )\binom{s_1}{2} \geq \mu \binom{s_1}{4},
\end{equation}
then $|N_G{(S)}| \geq \mu \binom{s_1}{4} \geq |S|$ and we are done. Inequality~\eqref{eq:4} holds if
$ s_{1}^2-5s_{1}-n^2+5n+12\alpha \leq 0 $ and then if $ s_{1}\leq \frac{1}{2}(5+\sqrt{{(2n-5)}^2 -48\alpha })$.

Now, suppose that $ s_1> \frac{1}{2}({5+\sqrt{{(2n-5)}^2 -48\alpha }})$. We may also assume that $S'=B\setminus N_G(S)\not=\emptyset$ and $S'_{1}=S'\cap E(K_n)\not=\emptyset$. Therefore,
\[
|S_{1}'| =\binom{n}{2} -\big| N_G{(S_1)} \cap E( K_n) \big| \leq \binom{n}{2}- \binom{s_1}{2} <  \binom{n}{2} -  \binom{\frac{ 5+\sqrt{{(2n-5)}^2 -48\alpha} }{2}}{2} \leq 32,
\]
where the last inequality holds for $ n\geq 33 $ and $ \alpha \leq 5$. Thus, $ { |S_{1}'| \leq 31}$. By Lemma~\ref{lem6}(iii), we have
\[
{ |N_G(S')|=|N_G{(S_{1}')}| \geq \mu |S_{1}'| \binom{n-32}{2}\geq  \mu \dfrac{|S'|}{\lambda_2}  \binom{n-32}{2} }\geq \binom{n-32}{2}\dfrac{|S'|}{\dfrac{(n-2)(n-3)}{12}+\dfrac{\alpha}{\mu}} \geq |S'|,
\]
where the last inequality holds for $ n\geq 54 $ and $ \alpha\leq 5 $. Therefore, $ |N_G(S)|\geq |S|$ and Hall's condition holds and thus, $ G $ contains a perfect matching.\\

 \textbf{Case 3. } $ 5\leq k\leq n-3. $

First, note that since $ k\geq 5$, for every $ K\in B $, $|{ N_G(\{K\})\cap E(K_n^{(k)})|= \binom{n-2}{k-2}\geq \binom{n-2}{3} \geq 2 \alpha \binom{n}{2}}$, where the last inequality holds for $ n\geq 38 $ and $ \alpha\leq 5 $. Thus, if ${\vert  S_1 \vert   > \binom{n}{k} - 2 \alpha \binom{n}{2}}$, then for every $ K\in B $, $ N_G(\{K\})\cap S_1\neq \emptyset $ and so $ {N_G(S)=B }$. Hence, in this case, clearly $ |N_G(S)|\geq |S| $ and we are done. 

Now, suppose that ${\vert  S_1 \vert \leq \binom{n}{k} - 2 \alpha \binom{n}{2}}$. 
Define $ a \in \mathbb{R}$ with $ 0< a\leq 1 $ such that $ |S_1| =a \binom{n}{k}= \binom{s_1}{k} $. Also, define $b\in \mathbb{R}$ such that $  | N_G{(S_1)} \cap E( K_n) | =b\binom{n}{2}$ and let $ g={\alpha \binom{n}{2} }/ \binom{n}{k}  $. By Lemma~\ref{lem6}(i), we have $ b\binom{n}{2}\geq \binom{s_1}{2} $. Therefore,
 \[ \frac{b^k}{a^2}\geq \frac{\binom{s_1}{2}^k\binom{n}{k}^2}{\binom{n}{2}^k\binom{s_1}{k}^2}\geq 1.\]
Note that the above inequality follows from the fact that the function $ \frac{x^2 (x-1)^2 \ldots (x-k+1)^2}{x^k (x-1)^k} $ is increasing when $ x\geq k $.
Therefore,  $ b\geq a^{2/k}$. Since ${ |S_1| \leq \binom{n}{k} - 2 \alpha \binom{n}{2}}$, we have
 \[ a \binom{n}{k}=|S_1| \leq \binom{n}{k} - 2\alpha \binom{n}{2}= (1-2g) \binom{n}{k} \leq  {(1-g)}^{2} \binom{n}{k}, \]
  and this implies that ${ a^{1/2}\leq 1-g }$ and since $ k\geq 4 $ and $ 0<a\leq 1 $, we have $  {a^{1-{2/k}} \leq  a^{1/2}\leq 1-g }$.
 Hence, by the definition of $ G $,
 \begin{align*}
   |N_G{(S)}|&=|N_G(S_1)| \geq \lambda_1 \vert N_G{(S_1)} \cap E( K_n) \vert =\lambda_1 b \binom{n}{2}  \geq \lambda_1 \binom{n}{2}  a^{2/k} \\
 &\geq    (\mu \binom{n}{k} -\alpha \binom{n}{2}) a^{2/k}\geq  \mu \binom{n}{k}  (1-g) a^{2/k}
 \geq \mu \binom{n}{k}  a^{1-2/k} a^{2/k }=\mu |S_1|\geq | S |.
 \end{align*}
Consequently, Hall's condition holds and thus, $ G $ contains a perfect matching, as desired.
  \end{proof}

\section{Proof of Theorem~\ref{lem:disjoint}} \label{sec:packing}

\textit{Proof of necessity.} Let $\mathcal{P}$ be an $M$-cycle packing for $G=\lambda K_n-I$ and $R(G)=E(G)\setminus \bigcup (E(C): C\in \mathcal{P})$. Since the induced subgraph of $G$ on the edges of the members of $ \mathcal{P}$ is even, the induced subgraph of $G$ on $R(G)$ is also even. Therefore, $r=|R(G)|\neq 1$ and if $r=|R(G)|=2$, then $R(G)$ contains two parallel edges. The condition (i) is trivial. Now, let $\lambda$ be odd. For every pair of distinct vertices $u,v$, among $\lambda$ parallel edges between $u$ and $v$, at least one edge does not contribute in the cycles of length two in $\mathcal{P}$. Thus, $2 \nu_2(M)\leq  (\lambda -1)\binom{n}{2}$. Now, if $r=2$, then, assuming $M'=(m_1,\ldots, m_t,2)$, $G$ has an $M'$-cycle decomposition and thus, using Theorem~\ref{thm2}, we have $2 \nu_2(M)<2 \nu_2(M')\leq  (\lambda -1)\binom{n}{2}$. When $\lambda$ is even, the condition (iii) immediately follows from Lemma~\ref{lem:packing}.

\textit{Proof of sufficiency.} For the case $r=0$, the assertion immediately follows from Theorem~\ref{thm2}. Now, suppose that $r\neq 0$. First, assume that $\lambda$ is odd. If $r=2$, then let $M'=(m_1,\ldots, m_t, 2)$. Since $\lambda$ is odd and $ 2 \nu_2(M)<  (\lambda -1)\binom{n}{2}$, we have $2\nu_2(M')= 2 \nu_2(M)+2\leq  (\lambda -1)\binom{n}{2}$. On the other hand, $\sigma(M')=f(\lambda,n)$. Hence, $M'$ is a $(\lambda,n)$-admissible list and by Theorem~\ref{thm2}, $\lambda K_n- I$ admits an $M'$-cycle decomposition. Now, removing a cycle of length two from this decomposition yields an $M$-cycle packing for $\lambda K_n -I$. If $r\geq 3$, then there exist some integers $m_{t+1}, \ldots, m_{t+s}\in \{3,4,5\}$, such that  $m_{t+1}+\cdots+m_{t+s}=r$. Let $M'=(m_1,\ldots, m_{t+s})$. Thus, $2 \nu_2(M')=2 \nu_2(M)\leq  (\lambda -1)\binom{n}{2}$ and $\sigma(M')=f(\lambda,n)$. Therefore, $M'$ is a $(\lambda,n)$-admissible list and by Theorem~\ref{thm2}, $\lambda K_n- I$ admits an $M'$-cycle decomposition. Now, removing the cycles of lengths $m_{t+1},\ldots, m_{t+s}$ from this decomposition yields the desired packing. 

Now, assume that $\lambda$ is even. 
If $r\leq m_1$, then define $m_{t+1}=r$ and $M'=(m_1,\ldots, m_{t+1})$.  Since $\sigma(M')=f(\lambda,n)$, $\max\{m_1,\ldots, m_{t+1}\}=m_1$ and $ m_1+(t+1)-2\leq ({\lambda}/{2})\binom{n}{2}$, $M'$ is a $(\lambda,n)$-admissible list and by Theorem~\ref{thm2}, $\lambda K_n$ admits an $M'$-cycle decomposition. Again, removing the cycle of length $m_{t+1}$ yields the desired packing. 
If $r\geq m_1+1$, then let  $m_{t+1}$ be a number in the set $\{m_1-1,m_1,m_1+1\}$ such that $2\leq m_{t+1}\leq n$ and $r-m_{t+1}$ is even (we leave the reader to check that such a number exists). Also, define $M'=(m_1,\ldots, m_{t+s})$ such that $ s\geq 1 $,   $m_{t+2}=\cdots=m_{t+s}=2$ and $\sigma(M')=f(\lambda,n)=\lambda \binom{n}{2}$. 
Hence,
\[\lambda \binom{n}{2}=\sigma(M') \geq m_1+m_{t+1}+2(t+s-2).\]
Since $\lambda$ is even,
\[ \frac{\lambda}{2}\binom{n}{2}\geq \left\lceil \frac{m_1+m_{t+1}}{2}\right \rceil+t+s-2=\max\{m_1,\ldots,m_{t+s}\}+t+s-2.\]
This implies that $M'$ is a $(\lambda,n)$-admissible list and thus, again by Theorem~\ref{thm2},  there is an $M'$-cycle decomposition of $\lambda K_n$. Removing the cycles of lengths $m_{t+1},\ldots, m_{t+s}$ from this decomposition yields the desired collection.

 \section{Proof of Theorem~\ref{thm5}}\label{sec:hypergraph}


 Let  $ P=(m_1,\ldots,m_s) $ and $ C=( n_1,\ldots,n_t) $ be the lists of integers satisfying the conditions of Theorem~\ref{thm5}. Without loss  of generality, we may assume that  $ C$ and $P $  are in a non-increasing order. \\

 \textbf{Case 1.} $3\leq k\leq n-3$.
Let $ \lambda $ and $\lambda'$  be  two integers such that
 \begin{equation}
  \lambda \binom{n}{2} \leq m_1+\cdots +m_s < (\lambda +1) \binom{n}{2}, \label{eq:6} 
\end{equation}
and
\begin{equation}
\lambda' \binom{n}{2} \leq n_1+\cdots +n_t < (\lambda' +1) \binom{n}{2}.
\end{equation}

We prove that there exist two multigraphs $H_P$ and $H_C$  on the vertex set $\{1,\ldots, n\}$ such that $H_P$ (resp. $H_C$) admits a $P$-path decomposition (resp. a $C$-cycle decomposition) and also, for every pair of distinct vertices $x,y$, the multiplicity of the edge $ xy $ in $ H_P $ (resp. $ H_C $) satisfies $\lambda\leq m_{H_P}(xy)\leq \lambda+1$ (resp. $\lambda'-2 \leq m_{H_C}(xy)\leq \lambda'+2$). First, assume that such multigraphs exist.  Now, let $H$ be the edge-disjoint union of $H_P$ and $H_C$ which is a multigraph on the vertex set $\{1,\ldots,n\}$. It is evident that $|E(H)|=\mu \binom{n}{k}$ and the edges of $H$ can be decomposed into $s$ paths of lengths $m_1,\ldots, m_s$ and $t$ cycles of lengths $n_1, \ldots, n_t$. On the other hand, for every pair of distinct vertices $x,y$, we have $\lambda+\lambda'-2\leq m_H(xy)\leq\lambda+\lambda'+3$. Therefore, by Lemma~\ref{thm:sdr}, there is a one-to-one correspondence $\eta: E(H)\to E(\mu K_n^{(k)}) $ such that $e\subset \eta(e)$, for every $e\in E(H)$. Hence, $\eta $ maps every path or cycle in  $H$ into a Berge path or a Berge cycle of the same length in $\mu K_n^{(k)}$. Consequently, $ \eta $ induces a desired decomposition of $\mu K_n^{(k)}$ into appropriate Berge paths and cycles and the proof completes. In the sequel, we show how one can construct the multigraphs $ H_P $ and $ H_C $ with the desired properties. \\

\textbf{Construction of $H_P$.} \\
Choose $ s_0$ to be the largest  integer such that $  m_1+\cdots+m_{s_0 }\leq \lambda\binom{n}{2}$ and let $q=\lambda\binom{n}{2} -(m_1+\cdots+m_{s_0 })$. If $ s_0<s $, then clearly $q<m_{s_0+1}$ and let $q'=m_{s_0+1}-q$. Also, if $ s_0=s $, then $ q=0 $ and let $ q'=0 $.  Define $P'=(m_1,\ldots, m_{s_0},q)$ and $P''=(q',m_{s_0+2},\ldots, m_s)$. It is clear that $m_1+\cdots+ m_{s_0}+q=\lambda \binom{n}{2}$ and $q'+m_{s_0+2}+\cdots+ m_s<\binom{n}{2}$. Thus, by Theorem~\ref{thm:path}, the multigraph $H_1= \lambda K_n$ on the vertex set $\{1,\ldots, n\}$ admits a $P'$-path decomposition $ \mathcal{P}' $ and there is a subgraph $H_2$ of $K_n$ on the vertex set $\{1,\ldots, n\}$ which admits a $P''$-path decomposition $ \mathcal{P}'' $. By renaming the vertices of $H_2$, we may assume that the path of length $q$ in $ \mathcal{P}' $ and the path of length $q'$ in $ \mathcal{P}'' $ have an edge-disjoint union equal to a path of length $q+q'=m_{s_0+1}$.  Now, define the multigraph $H_P$ as the edge-disjoint union of $H_1$ and $H_2$. Evidently, $H_P$ admits a $P$-path decomposition and for every pair of distinct vertices $x,y$, we have $\lambda\leq m_{H_P}(xy)\leq \lambda+1$.\\

\textbf{Construction of $H_C$.} \\
If $\lambda'(n-1)$ is odd, then let $I\subset E(\lambda' K_n)$  be a perfect matching of $\lambda'K_n$ and otherwise, let $I$ be the empty set. Also, let $f(\lambda',n)$ be the number of edges of $\lambda' K_n-I$. 
Let $ t_0$ be the largest  integer such that $  n_1+\cdots+n_{t_0 }\leq f(\lambda',n)$ and let $r=  f(\lambda',n) -(n_1+\cdots+n_{t_0 })$. If $ t_0<t $, then clearly  $ 0  \leq r <n_{t_0+1} $ and if $ t_0=t $, then $ r=0 $.
Now, define the list of integers $M$ as follows.

\begin{itemize}
\item[(i)] If $ r=0 $,  then let  $ M=(n_1,\ldots,n_{t_0}) $.
\item[(ii)] If $ r=1 $, then let $M=(n_1,\ldots,n_{t_0-1},n_{t_0}+1 )$. Note that in this case,  $ n_{t_0} < n $ (since if $ n_{t_0}=n $, then $ n_1=\cdots=n_{t_0}=n $ and so $ r $ is a multiple of $ n $).
\item[(iii)] If $ r\geq 2 $, then  let  $ M=(n_1,\ldots,n_{t_0},r) $. Note that in this case, $ \nu_2(M)\leq 1 $ (because $ n_{t_0}\geq n_{t_0+1}>r\geq 2 $).
\end{itemize}

It is clear that $ \sigma(M)= f(\lambda',n) $. Now, we construct $ H_C $ in the following two cases. \\

\textbf{Case 1.1.}  $ \lambda' $ is odd.

First, suppose that $ \nu_2(C) \geq \frac{\lambda'-1}{2}\binom{n}{2} $.
Then, the complete multigraph $H_3=(\lambda'-1) K_n$ obviously admits a decomposition into $\frac{\lambda'-1}{2}\binom{n}{2}$ cycles of length two. Let $t'=t-\ds\frac{\lambda'-1}{2}\binom{n}{2}$. Since
\[\sum_{i=1}^{t'} n_i\leq (\lambda'+1) \binom{n}{2}-\sum_{i=t'+1}^t n_i=2 \binom{n}{2},  \]
by Theorem~\ref{lem:disjoint}, the multigraph $ 3K_n-I $ contains $t'$ pairwise edge-disjoint cycles of lengths $n_1,\ldots,n_{t'}$. Now, let $H_4$ be the union of these cycles. Also, let $H_C$ be the edge-disjoint union of $H_3$ and $H_4$. Hence, for every pair of distinct vertices $x,y$, we have $\lambda'-1\leq m_{H_C}(xy)= m_{H_3}(xy)+m_{H_4}(xy)\leq 3+\lambda'-1=\lambda'+2$. Also, $H_C$ admits a $C$-cycle decomposition.

Now, suppose that  $ \nu_2(C) <\frac{\lambda'-1}{2}\binom{n}{2} $. Therefore, $  \nu_2(M) \leq \frac{(\lambda' -1)}{2}\binom{n}{2} $, and so $M$ is $(\lambda', n)$-admissible. Then, by Theorem~\ref{thm:cycle}, $ \lambda' K_n-I $ admits an $ M $-cycle decomposition. If $r\geq 2$ (resp. $r=1$), then let $F$ be the edge set of a cycle of length $r$ (resp. $n_{t_0}+1$) in the $M$-cycle decomposition of $\lambda'K_n$ (if $r=0$, then let $F$ be the empty set). Let $H_3$ be the multigraph obtained from $\lambda'K_n-I$ by removing all edges in $F$. Clearly, for every pair of distinct vertices $x,y$, we have $\lambda'-2\leq m_{H_3}(xy)\leq \lambda'$.
On the other hand,
\begin{align*}
\sum_{i=t_0}^{t} n_i &\leq 2n+ \sum_{i=t_0+2}^{t} n_i\leq 2n+(\lambda'+1)\binom{n}{2}- \sum_{i=1}^{t_0+1} n_i \\ &\leq 2n+(\lambda'+1)\binom{n}{2}-f(\lambda',n)= 2n+\binom{n}{2}+|I|<2\binom{n}{2}-2n,
\end{align*}
where the last inequality holds for $ n\geq 11 $. Hence, by Theorem~\ref{lem:disjoint}, there are $t-t_0+1$ pairwise edge-disjoint cycles of lengths $n_{t_0},n_{t_0+1},\ldots,n_{t}$ in $2K_n$.
If $r\neq 1$ (resp. $r=1$), then let $H_4$ be the union of these cycles of lengths  $n_{t_0+1},\ldots,n_{t}$ (resp. $n_{t_0},n_{t_0+1},\ldots,n_{t}$).
Thus, for every pair of distinct vertices $x,y$, we have $0\leq m_{H_4}(xy)\leq 2$.
Now, define the multigraph $H_C$ as the edge-disjoint union of $H_3$ and $H_4$. Evidently, $H_C$ admits a $C$-cycle decomposition and for every pair of distinct vertices $x,y$, we have $\lambda'-2\leq m_{H_C}(xy)\leq \lambda'+2$. \\

\textbf{Case 1.2.}  $ \lambda' $ is even.

Let $l$ be the size of $M$ which is equal to $t_0$ if $r\leq 1$ and is equal to $t_0+1$ if $r\geq 2$. First, suppose that  $  n_1+l-2\geq \frac{\lambda'}{2} \binom{n}{2}$.
 Then, we have
\[
\lambda'\binom{n}{2}= \sigma(M) \geq n_1+2\nu_2(M)+ 3(l-\nu_2(M)-1) \geq \frac{3\lambda'}{2}\binom{n}{2} -2n_1 -\nu_2(M) +3.
\]
Therefore,
\[\nu_2(C)\geq \nu_2(M)-1\geq \frac{\lambda'}{2} \binom{n}{2} -2n+2 \geq \frac{\lambda'-2}{2}\binom{n}{2}. \]

Clearly the complete multigraph $H_3=(\lambda'-2) K_n$ can be decomposed into $\ds\frac{\lambda'-2}{2}\binom{n}{2}$ cycles of length two. Now, let $t'=t-\ds\frac{\lambda'-2}{2}\binom{n}{2}$. Since
\[\sum_{i=1}^{t'} n_i\leq (\lambda'+1) \binom{n}{2}-\sum_{i=t'+1}^t n_i=3 \binom{n}{2},  \]
by Theorem~\ref{lem:disjoint}, the multigraph $ 4K_n $ contains $t'$ pairwise edge-disjoint cycles of lengths $n_1,\ldots,n_{t'}$. Now, let $H_4$ be the union of these cycles. Also, let $H_C$ be the edge-disjoint union of $H_3$ and $H_4$. Hence, for every pair of distinct vertices $x,y$, we have $\lambda'-2\leq m_{H_C}(xy)= m_{H_3}(xy)+m_{H_4}(xy)\leq 4+\lambda'-2=\lambda'+2$. Also, $H_C$ admits a $C$-cycle decomposition.

Now, suppose that $n_1+l-2< \frac{\lambda'}{2} \binom{n}{2}$. Hence, $M$ is $(\lambda', n)$-admissible and then, by Theorem~\ref{thm:cycle}, $ \lambda' K_n-I $ admits an $ M $-cycle decomposition. If $r\geq 2$ (resp. $r=1$), let $F$ be the edge set of a cycle of length $r$ (resp. $n_{t_0}+1$) in the $M$-cycle decomposition of $\lambda'K_n-I$ (if $r=0$, then let $F$ be the empty set). Let $H_3$ be the multigraph obtained from $\lambda'K_n-I$ by removing all edges in $F$. Clearly, for every pair of distinct vertices $x,y$, we have $\lambda'-2\leq m_{H_3}(xy)\leq \lambda'$ (since $I$ is empty).

On the other hand,
\begin{align*}
\sum_{i=t_0}^{t} n_i &\leq 2n+ \sum_{i=t_0+2}^{t} n_i\leq 2n+(\lambda'+1)\binom{n}{2}- \sum_{i=1}^{t_0+1} n_i \\ &\leq 2n+(\lambda'+1)\binom{n}{2}-f(\lambda',n)= 2n+\binom{n}{2}<2\binom{n}{2}-2n,
\end{align*}
where the last inequality holds for $ n\geq 10 $. Hence, by Theorem~\ref{lem:disjoint}, the complete multigraph $ 2K_n $ contains $t-t_0+1$ pairwise edge-disjoint cycles of lengths $n_{t_0},n_{t_0+1},\ldots,n_{t}$.
If $r\neq 1$ (resp. $r=1$), then let $H_4$ be the union of these cycles of lengths  $n_{t_0+1},\ldots,n_{t}$ (resp. $n_{t_0},n_{t_0+1},\ldots,n_{t}$).
Thus, for every pair of distinct vertices $x,y$, we have $0\leq m_{H_4}(xy)\leq 2$.
Now, define the multigraph $H_C$ as the edge-disjoint union of $H_3$ and $H_4$. Evidently, $H_C$ admits a $C$-cycle decomposition and for every pair of distinct vertices $x,y$, we have $\lambda'-2\leq m_{H_C}(xy)\leq \lambda'+2$. \\

\textbf{Case 2.} $k=n-2$. \\
Consider the complete multigraph $ \mu K_n $ on the vertex set $ \{1,\ldots,n \} $ and let $ c $ be a proper edge coloring of $\mu K_n$ with $\chi'=\mu(2\lfloor (n-1)/2\rfloor +1)$ colors $ \{1,2,\ldots, \chi'\} $, where each color class is of size exactly $\lfloor n/2\rfloor$. Let $E=(e_1,e_2,\ldots ,e_{\mu \binom{n}{2}})$ be an ordering of the edge set of $\mu K_n$ such that $c(e_p)\leq c(e_q)$, for every $p<q$. For every $i\in\{1,\ldots, \mu \binom{n}{2}\}$, let $f_i=\{1,\ldots, n\}\setminus e_i$.
Let $\sigma_0=0$. Also, for every $i\in\{1,\ldots, t\}$,  let $\sigma_i= \sum_{j=1}^i n_j$, and for every $i\in\{t+1,\ldots,  t+s\}$, let $\sigma_i=\sigma_t+\sum_{j=1}^{i-t} m_j$.  Now, for every $i\in\{1,\ldots, t+s\}$, define $E_i=\{f_{\sigma_{i-1}+1}, \ldots, f_{\sigma_i}\}$ and $E'_i=\{e_{\sigma_{i-1}+1}, \ldots, e_{\sigma_i}\}$. It is clear that $\{E_1,E_2,\ldots, E_{t+s}\}$ is a partition of the edge set of $\mu K_n^{(k)}$. Now, for every $i\in\{1,\ldots, t\}$ (resp. $i\in \{t+1,\ldots, t+s\}$), we construct a Berge cycle (resp. Berge path) of length $n_i$ (resp. $m_i$) with the edges in $E_i$. This gives the desired decomposition of $ \mu K_n^{(k)} $.

To do this, fix $i\in\{1,\ldots, t\}$ and for every $j\in\{1,\ldots, n_i-1\}$, let $g_{j+1}=f_{\sigma_{i-1}+j}\cap f_{\sigma_{i-1}+j+1}$. Also, let $g_{1}=f_{\sigma_i}\cap f_{\sigma_{i-1}+1}$. It is clear that $|g_j|\geq n-4$, for every $j\in \{1,\ldots, n_i\}$. We claim that $(g_1,\ldots g_{n_i})$ has an SDR $(x_1,\ldots, x_{n_i})$ (i.e. $x_j\in g_j$ and $x_j$'s are distinct).
Since $n_i\leq n$ and $c(e_p)\leq c(e_q)$, for every $p<q$, at most three colors appear on the edges in $E'_i$. Therefore, every $j\in \{1,\ldots,n\}$ is a member of at least $n_i-3$ sets in $E_i$ and thus, for every $j\in \{1,\ldots,n\}$,
at least $n_i-6$ sets in $\{g_1,\ldots, g_{n_i}\}$ contain $j$. Therefore, the union of every 7 sets in $\{g_1,\ldots, g_{n_i}\}$ is equal to the whole set $\{1,\ldots, n\}$. Now, we check Hall's condition for $(g_1,\ldots, g_{n_i})$. Let $I\subseteq \{1,\ldots, n_i\}$. If $1\leq |I|\leq n-4$, then $|\cup_{j\in I}g_j|\geq n-4\geq |I|$. Also, if $|I|\geq n-3$, then since $n\geq 10$, by the above argument, $\cup_{j\in I} g_j=\{1,\ldots, n\}$ and thus, $|\cup_{j\in I} g_j|=n\geq  |I|$.
Hence, by Hall's theorem, $(g_1,\ldots, g_{n_i})$ has an SDR $(x_1,\ldots, x_{n_i})$ and so $E_i$ is the edge set of a Berge cycle of length $n_i$ in $ \mu K_n^{(k)} $ with the core sequence $(x_1,\ldots, x_{n_i})$.
Now, fix $i\in\{t+1,\ldots, s+t\}$ and for every $j\in \{1,\ldots, m_i-1\}$, let $g_{j+1}=f_{\sigma_{i-1}+j}\cap f_{\sigma_{i-1}+j+1}$. Also, let $g_1=f_{\sigma_{i-1}+1}$ and $g_{m_i+1}=f_{\sigma_i}$.  By a similar argument, $(g_1,\ldots, g_{m_i+1})$ has an SDR $(x_1,\ldots, x_{m_i+1})$ and thus, $E_i$ is the edge set of a Berge path of length $m_i$ in $ \mu K_n^{(k)} $ with the core sequence $(x_1,\ldots, x_{m_i+1})$. This completes the proof for the case $k=n-2$. \\

\textbf{Case 3.} $k=n-1$. \\
For each $i\in\{1,\ldots, \mu n\}$, let  $r_i=(i-1 \mod n)+1$ and $f_{i}=\{1,\ldots, n\} \setminus \{r_i\}$. Also, let $E=(f_1,\ldots, f_{\mu n})$ which is an ordering of the edge set of $\mu K_n^{(k)}$. Now, for every $i\in\{1,\ldots, t+s\}$, define $E_i=\{f_{\sigma_{i-1}+1}, \ldots, f_{\sigma_i}\}$, where $ \sigma_i$'s are chosen as in Case 2. It is clear that $\{E_1,E_2,\ldots, E_{t+s}\}$ is a partition of the edge set of $\mu K_n^{(k)}$. Also, for every $i\in\{1,\ldots, t\}$, if $n_i\geq 3$ (resp. $n_i=2$), then $E_i$ is the edge set of a Berge cycle in $ \mu K_n^{(k)} $ with the core sequence $(r_{\sigma_{i-1}+2}, \ldots, r_{\sigma_i}, r_{\sigma_{i}+1})$ (resp. $(r_{\sigma_{i-1}+3},r_{\sigma_{i-1}+4})$) and for every $i\in\{t+1,\ldots, t+s\}$, $E_i$ is the edge set of a Berge path in $ \mu K_n^{(k)} $ with the core sequence $(r_{\sigma_{i-1}+2}, \ldots, r_{\sigma_{i}+1}, r_{\sigma_{i}+2})$, as desired.



\end{document}